\documentclass[11pt]{amsart}
\usepackage{lmodern}  
\usepackage{anyfontsize}
\usepackage{geometry}
\usepackage{fullpage}
\usepackage{parskip}
\usepackage[pdftex]{lscape}
\usepackage{euscript}
\usepackage[all]{xy}
\usepackage{graphicx, xcolor}
\usepackage{comment}
\usepackage[alphabetic,msc-links]{amsrefs}

\usepackage{mathrsfs, hyperref}
\usepackage{enumerate,float,makecell, multicol,longtable}

\newtheorem{thm}{Theorem}
 
\newtheorem{lem}[thm]{Lemma} 
\newtheorem{prop}[thm]{Proposition}

\theoremstyle{definition}

\newcommand{\op}{\emph{op.\ cit.}\ }

\newcommand{\PP}{\mathbb{P}}

\newcommand{\ra}{\rightarrow}
\newcommand{\bC}{\mathbb{C}}
\newcommand{\fB}{\mathfrak{B}}

\newcommand{\GL}{\mathrm{GL}}
\newcommand{\ft}{\mathfrak{t}}

\newcommand{\cC}{\mathcal{C}}
\newcommand{\cK}{\mathcal{K}}
\newcommand{\cD}{\mathcal{D}}
\newcommand{\cM}{\mathcal{M}}

\newcommand{\Irr}{\mathrm{Irr}}

\newcommand{\Conn}{\mathrm{Conn}}
\newcommand{\bP}{\mathbb{P}}
\newcommand{\Fq}{\mathbb{F}_q}

\newcommand{\cF}{\mathcal{F}}
\newcommand{\fg}{\mathfrak{g}}

\begin{document}
\title{Counting points on braid varieties\\ and the Deligne--Simpson problem} 
\author{Masoud Kamgarpour}
\author{Bailey Whitbread}

\address{School of Mathematics and Physics, The University of Queensland, Australia}
\email{\href{mailto:masoud@uq.edu.au}{masoud@uq.edu.au}}

\address{School of Mathematics and Statistics, The University of Sydney, Australia}
\email{\href{mailto:b.whitbread@usyd.maths.edu.au}{b.whitbread@usyd.maths.edu.au}}

\begin{abstract} 
We solve the isoclinic Deligne–Simpson problem for exceptional groups, completing a program initiated by Sage et al.\ and Jakob--Yun. As a by-product, we obtain new examples of physically rigid irregular $G$-connections on $\bP^1$. Our approach uses the Riemann--Hilbert correspondence to reduce the problem to determining the non-emptiness of certain braid varieties associated to periodic braids. We show that this can be achieved by counting points over finite fields. Our approach is inspired by Lusztig's construction of a map from conjugacy classes in the Weyl group to unipotent  classes. 
\end{abstract} 

\maketitle 
\section{Introduction}  

Let $G$ be a complex reductive group. The Deligne--Simpson problem concerns the existence of $G$-connections on curves with prescribed singularities at finitely many points. For $G=\GL_n$, substantial progress has been made using non-abelian Hodge theory \cite{Simpson}, $\ell$-adic cohomology \cite{Katz}, and quiver varieties \cite{CB, Hiroe}. In contrast, much less is known for groups of general type.  

In this paper, we solve the Deligne--Simpson problem for a class of $G$-connections on $\bP^1-\{0,\infty\}$ known as \emph{isoclinic}.\footnote{Our terminology is consistent with \cite{JY} but different from \cite{SageReview}. In particular, what we call isoclinic is called toral in \cite{SageReview}. Following \cite{HJ}, we reserve the term toral for a more general notion; see \S \ref{s:Toral}.} Such connections have an irregular singularity at $0$ with a regular semisimple Levelt--Turrittin leading term, and a regular singularity at $\infty$. 
Important examples include generalised Frenkel--Gross connections (also known as generalised Kloosterman connections) \cite{FrenkelGross, HeinlothNgoYun}, Yun's $\theta$-connections \cite{YunEpipelagic, ChenRigid}, and Coxeter connections \cite{KS-rigid, KS-galois}.   

Suppose $G$ is a simple group and let $W$ denote its Weyl group. Then it is known that the slope of an isoclinic connection is of the form $\nu=d/m$, where $m$ is a regular number for $W$ (in the sense of Springer \cite{Springer}) and $d$ is a positive integer coprime to $m$, cf. 
 \cite[Lemma 2.2]{JY}. Our main result is as follows: 

\begin{thm} \label{t:main} There exists a (necessarily unique) unipotent conjugacy class $\mathcal{C}_\nu\subset G$ such that there is an isoclinic $G$-connection of slope $\nu$ and unipotent monodromy $\cC$ if and only if $\cC \ge \mathcal{C}_\nu$.\footnote{Here $\geq$ denotes the usual partial order on unipotent conjugacy classes: $\cC\geq \cC'$ if and only if $\overline{\cC}\supseteq \cC'$.} When $\nu>1$, we have $\mathcal{C}_\nu=\{1\}$. When $\nu<1$ and $G$ is of exceptional type, the classes $\mathcal{C}_\nu$ are listed in \S \ref{ss:minimal}.
\end{thm}

When $G=\GL_n$ and $m=n$, this result was established by Sage and collaborators \cite{SageEtAl, LivesaySageNguyen} using intricate linear algebra. Jakob and Yun \cite{JY} proved the theorem for classical groups via non-abelian Hodge theory. They also treated some exceptional cases; however, groups of type $E$ remained unresolved. Our approach for exceptional groups is different and uses instead the Riemann--Hilbert correspondence. 

The version of the Riemann--Hilbert correspondence we need was established recently by Hohl and Jakob \cite{HJ}, building on prior work of Boalch \cite{BoalchTopology}. The key input is that irregular connections give rise to positive braids \cite{BBAMY}. The moduli space of irregular connections is then identified with the corresponding braid variety. 
These varieties first appeared in the work of Deligne on the $K(\pi,1)$ Conjecture \cite{Deligne}, and have recently attracted considerable interest due to their central role in symplectic topology, knot invariants, and cluster algebras \cite{STZ, Trinh, GorskyKivinenSimental, IanEtAl, Mellit, GLSB}. To take into account the monodromies of connections, we need to work with variants of braid varieties introduced by Trinh \cite{Trinh}. 

It turns out that the class of braids arising from isoclinic connections coincide with \emph{periodic} braids, i.e., those with a non-trivial central power. 
More precisely, we show that the braid associated to an isoclinic connection of slope $\nu=d/m$ is $(\widetilde{w_m})^d$, where $w_m$ is an \emph{$m$-Springer element} (in the sense of \cite{BM}), and $\widetilde{w_m}$ denotes its lift to the positive braid monoid. Deep results of Brou\'{e}--Michel and Bessis \cite{BM, BessisKPi1} imply that, up to conjugacy in the braid group, $\widetilde{w_m}$ is the unique $m^{\text{th}}$-root of the full twist.

The above discussion implies that the existence of an isoclinic connection of slope $\nu$ and monodromy $\cC$ is equivalent to the non-emptiness of the braid variety associated to $(\widetilde{w_m})^d$ and $\cC$. We shall see that this in turn is equivalent to the condition 
\[
(Bw_mB)^d\cap \cC\neq \emptyset. 
\]

The problem of determining when the intersection of a Bruhat double coset with a conjugacy class is non-empty  has been studied extensively and is known to be difficult even for $\GL_n$, cf.\ \cite{EG04, EG07, CLT10, Lusztig, LusztigCSmall, Malten}. When  $w$ is an element of minimal length in its conjugacy class, Lusztig proved that there exists a (necessarily unique) unipotent conjugacy class $\cC_w$ such that a unipotent class $\cC$ satisfies $BwB\cap \cC\neq \emptyset$ if and only if $\cC\geq \cC_w$. This yields Lusztig's map 
\[
\begin{aligned}
\{\textrm{conjugacy classes in } W\} 
  &\longrightarrow \{\textrm{unipotent conjugacy classes in } G\} \\
[w] &\longmapsto \cC_w.
\end{aligned}
\]
For groups of exceptional type, Lusztig's proof relied on counting points over finite fields \cite{Lusztig}. As he observed, the count reduces to  character values of the finite reductive group $G(\Fq)$ and the finite Hecke algebra $H_q$.
 
  Our proof of Theorem~\ref{t:main} for exceptional groups follows the same strategy; that is, we count points on the relevant braid varieties (more precisely, braid stacks) over finite fields. As in Lusztig's case, the count reduces to character theory of $G(\Fq)$ and $H_q$. The necessary character  computations are carried out with the CHEVIE package for Julia developed by Michel \cite{Michel}. Note that $m$-Springer elements and their powers do not usually have minimal length in their conjugacy classes. Moreover, in general, we need to take powers of Bruhat double cosets. Thus, our main theorem is not a formal consequence of Lusztig's result. 

As a byproduct of Theorem \ref{t:main}, we obtain new examples of cohomologically rigid  isoclinic connections; see \S \ref{ss:rigidTable}. We show that these connections are physically rigid by proving that the corresponding braid varieties are singletons. Constructing the $\ell$-adic analogues of these connections and their Hecke eigensheaves is an interesting open problem.

Finally, it is natural to ask for explicit realisations of the isoclinic connections whose existence is guaranteed by our main theorem. We make partial progress in this direction by establishing Conjecture 5.7.(2) of \cite{SageReview}, concerning Coxeter $G$-connections with minimal monodromy.

\subsection{Structure of the text}
In \S \ref{s:Toral}, we introduce toral connections (which are a generalisation of isoclinic connections), and discuss the Riemann--Hilbert correspondence and the Deligne--Simpson problem for them. We also recall Lusztig's theorem and outline his proof. 

In \S \ref{s:isoclinic}, we focus on toral connections that are isoclinic. We first give an explicit description of their associated braids. Then, we prove Theorem \ref{t:main} by counting points on the corresponding braid stacks over finite fields. We conclude by discussing rigid isoclinic connections and explicit realisations of Coxeter connections. 

In the appendix, we give tables listing minimal unipotent classes for isoclinic connections and rigid connections. We also give sample computations with the code used to establish Theorem \ref{t:main}.

\subsection{Acknowledgements} We thank Konstantin Jakob, Ian Le, Dan Sage,  and Minh-T\^{a}m Trinh for helpful discussions. 
We are grateful to Willem A. de Graaf and Jean Michel for answering our questions about GAP and CHEVIE. Finally, we thank the Sydney Mathematical Research Institute for providing excellent working conditions and the Australian Research Council for financial support.

\section{Toral connections} \label{s:Toral}

Let $G$ be an almost simple complex algebraic group, $B\subset G$ a Borel subgroup, $T\subset B$ a maximal torus, $W$ the Weyl group, $\Phi$ the root system, and $\Phi^+$ the set of positive roots. Let $\fg$ and $\ft$ denote the Lie algebras of $G$ and $T$, respectively.

\subsection{Basic definitions}

For each positive integer $m$, let $\cK_m=\bC(\!(t^{1/m})\!)$, and write $\cK=\cK_1$. Let $\cD^\times=\mathrm{Spec}(\cK)$ denote the formal punctured disk. Let $\nabla$ be a formal $G$-connection, i.e., a connection on a $G$-bundle on $\cD^\times$. By a classical theorem of Turrittin and Levelt (cf. \cite{BabbittVaradarajan}), there exists a positive integer $m$ such that $\nabla$ is $G(\cK_m)$-gauge equivalent to a connection of the form
\begin{equation}\label{eq:connection}
d+(A(t^{1/m})+\cdots)\frac{dt}{t}.
\end{equation}
Here
\begin{equation} \label{eq:irregular} 
A(t^{1/m})=A_d\,t^{-d/m}+ A_{d-1}\, t^{-(d-1)/m} + \cdots + A_1t^{-1/m}\in t^{-1/m}\ft[t^{-1/m}],
\end{equation}
 $A_d$ is non-nilpotent, and the symbol $\cdots$ in \eqref{eq:connection} denotes an element of $\fg(\bC([\![t]\!]))$.

 The polynomial $A=A(t^{1/m})$ is called an \emph{irregular class} associated to $\nabla$, the element $A_d\in \ft$ is a \emph{Levelt--Turrittin leading term}, and the rational number $d/m$ is the \emph{slope} of $\nabla$.

An  element  $A(t^{1/m}) \in t^{-1/m}\ft[t^{-1/m}]$ is called an \emph{irregular class} if it arises as the irregular class of a formal connection on $\cD^\times$. Following \cite{HJ}, we call $A(t^{1/m})$ \emph{toral} if the only Weyl group element fixing $A_d, A_{d-1}, \cdots, A_1$ is the trivial element. Likewise, a formal $G$-connection is called \emph{toral} if its irregular class is toral. This condition is independent of the choice of $m$ and the gauge transformation.

  \subsection{Riemann--Hilbert Correspondence} Let $A=A(t^{1/m})\in \ft[t^{-1/m}]$ be a toral irregular class and $\cC\subset G$ a conjugacy class. 
  Let $\Conn(A,\cC)$ denote the stack of $G$-connections on $\bP^1-\{0,\infty\}$ with an irregular singularity at $0$ of irregular class $A$, and a regular singularity at $\infty$ with monodromy $\cC$. Our aim is to describe the Betti analogue of this stack.

Let $B_W^+$ denote the positive braid monoid. As explained in \cite[\S 4.3.3]{BBAMY}, the element $A(t^{1/m})$ determines a positive braid $\beta_A\in B_W^+$, unique up to a cyclic shift. We now define the corresponding braid stack. 

Let $\fB$ denote the flag variety of $G$. 
Given $w\in W$, we write $\widetilde{w}\in B_W^+$ for its positive lift. 
For a positive braid $\beta=\widetilde{w_1}\cdots \widetilde{w_d}\in B_W^+$, define the corresponding braid stack by
\[
\cM(\beta,\cC):=\bigg\{
(g,F_1,\dots,F_{d+1})\in \cC\times \fB^{d+1}
\,\bigg|\,
gF_1=F_{d+1}, \, \, w(F_i,F_{i+1})=w_i,\; i=1,\ldots,d
\bigg\}\bigg/G,
\]
where $w(F,F')\in W$ denotes the relative position of the flags $F$ and $F'$. (The associated braid variety is defined by taking the GIT quotient instead.) One readily checks that 
\[
\cM(\widetilde{w}, \cC) = (BwB\cap \cC)/B.
\]

\begin{thm} 
There is an equivalence of groupoids
$\Conn(A,\cC)(\bC)\simeq \cM(\beta_A,\cC)(\bC)$. 
\end{thm}

For $G=\GL_n$, this follows from results of Boalch \cite{BoalchTopology}. For general $G$, see \cite{HJ}, Theorem 3.5.3 and Proposition 4.2.2. 
It is expected that the above equivalence can be upgraded to an isomorphism of complex analytic stacks, cf. \cite[Conjecture 4.4.4]{BBAMY}.

\subsection{The Deligne--Simpson problem} 
Let $A \in t^{-1/m}\ft[t^{-1/m}]$ be an irregular class and $\cC\subset G$ a conjugacy class. 
The corresponding Deligne--Simpson problem  asks when $\Conn(A,\cC)$ is non-empty.
By the above discussion, when $A$ is toral, this problem is equivalent to determining when  $\cM(\beta_A,\cC)$ is non-empty.

We now give a group-theoretic criterion for the non-emptiness of $\cM(\beta,\cC)$ for a general positive braid $\beta=\widetilde{w_1}\cdots \widetilde{w_n}\in B_W^+$. Let
\[
B(\beta):=\prod_{i=1}^n (Bw_iB)\subseteq G.
\]
\begin{prop}\label{p:group}
We have $\cM(\beta,\cC)\neq\emptyset$ if and only if $B(\beta)\cap \cC\neq\emptyset$.
\end{prop}

\begin{proof}
Suppose $B(\beta)\cap \cC\neq\emptyset$ and choose an element
\[
g=b_1w_1b_2w_2\cdots b_nw_nb_{n+1}\footnote{Here, we have slightly abused notation by writing $w_i$ instead of a chosen lift of $w_i$ to $G$.}
\]
in this intersection. 
Then
\[
(g,\, B,\, b_1w_1B,\, b_1w_1b_2w_2B,\,\dots,\, gB)\in \cM(\beta,\cC).
\]

Conversely, suppose $\cM(\beta,\cC)\neq\emptyset$ and choose $(g,F_1,\dots,F_{n+1})$ in this stack. 
Using the $G$-action, we may assume $F_1=B$. 
The condition $w(F_1,F_2)=w_1$ implies $F_2=b_1w_1B$ for some $b_1\in B$. 
Similarly, the condition $w(F_2,F_3)=w_2$ implies $F_3=b_1w_1b_2w_2B$ for some $b_2\in B$. 
Continuing in this way, we obtain $F_{n+1}=gB$, where 
\[
g=b_1w_1\cdots b_nw_nb_{n+1},\qquad b_i\in B.
\] 
By assumption $g\in \cC$; thus, we have
$g\in \Big(\prod_{i=1}^n Bw_iB\Big)\cap \cC$, 
as required.
\end{proof}

For future reference, we record the following simple result: 
\begin{lem} \label{l:Full}
Suppose $\beta\in B_W^+$ contains the full twist $\Pi_W$. Then $B(\beta)=G$; in particular, $B(\beta)\cap \cC$ is non-empty for every conjugacy class $\cC\subset G$. 
\end{lem} 

\begin{proof} It is sufficient to show that 
\[
B(\Pi_W) := (Bw_0 B)(Bw_0 B) = G. 
\]
This follows from the fact that in a connected algebraic group $G$, the product $UV$ of two dense open subsets equals the whole group. Indeed, for every $g\in G$, $gV^{-1}$ and $U$ are also dense open subsets; thus, their intersection is non-empty. This implies that $g=uv$ for some $u\in U$ and $v\in V$. 
\end{proof}

\subsection{Nice braids}\label{ss:Nice}
Let us call a braid $\beta\in B_W^+$ \emph{nice} if there exists a (necessarily unique) unipotent class $\cC_\beta\subset G$ such that for any unipotent class $\cC\subset G$, we have $\cM(\beta, \cC)\neq \emptyset$ if and only if $\cC\geq \cC_\beta$. It seems reasonable to conjecture that braids arising from toral irregular classes are nice. Our main theorem amounts to the statement that braids arising from {isoclinic} irregular classes are nice. This will be discussed further below. In this subsection, we recall a related result of Lusztig. 

\begin{thm}[\cite{Lusztig}] 
Suppose $w\in W$ has minimal length in its conjugacy class. Then $\widetilde{w}$ is nice.
\end{thm}

Let us outline Lusztig's proof for groups of exceptional type. Let $k$ be an algebraic closure of a finite field $\Fq$, and consider an almost simple algebraic group over $k$ with the same root system as $G$. By abuse of notation, we denote this group by $G$ as well. Assume that the characteristic of $k$ is a good prime for $G$. Then, the unipotent classes of $G(k)$ are in bijection with the unipotent classes of $G(\bC)$, cf.\ \cite{Premet}.

Now, choose an $\Fq$-split rational structure on $G$, and let $F:G\to G$ denote the Frobenius map. Let $H_q$ be the finite Hecke algebra. The space of functions $\cF=\mathbb{C}[\fB^F]$ then carries the structure of a $(\mathbb{C}[G^F],H_q)$-bimodule. We write $(g,T_w)$ for the operator $f\mapsto g\cdot f\cdot T_w$, where $f\in \cF$.

Recall that there is a canonical bijection between irreducible representations of $W$ and irreducible unipotent principal series representations of the finite group $G^F$. For $E\in \Irr(W)$, we let $\rho_E$ denote the corresponding unipotent representation of $G^F$. Likewise, there is a canonical bijection between irreducible representations of $H_q$ and irreducible representations of $W$; we denote by $E_q$ the representation of $H_q$ corresponding to $E$.

In \cite[\S 1.2]{Lusztig}, Lusztig observes that
\begin{equation} \label{eq:Lusztig}
\big|\big((BwB\cap \cC)/B\big)^F\big|
= \frac{1}{|G^F|} \sum_{g\in \cC^F} \mathrm{tr}((g,T_w),\cF)
= \frac{1}{|G^F|} \sum_{g\in \cC^F}\sum_{E\in \mathrm{Irr}(W)}
\mathrm{tr}(T_w,E_q)\,\mathrm{tr}(g,\rho_E).
\end{equation}

For groups of exceptional types, Lusztig explicitly computes the right-hand side using a computer and determines when it is non-zero.

 Using \eqref{eq:Lusztig}, one can also find elements $w\in W$ such that $\widetilde{w}$ is not nice. One example is when $G$ is of type $F_4$ and $w$ is $s_2 s_3 s_2 s_4 s_3 s_2 s_3$; see \S\ref{ss:examplecode}. Note that this is a reduced expression for $w$; however, $w$ does not have minimal length in its conjugacy class, since it is conjugate to $s_4s_3s_2$.

\section{Isoclinic connections} \label{s:isoclinic}

An irregular class $A=A_d\, t^{-d/m}+\cdots +A_1t^{-1/m} \in t^{-1/m}\ft[t^{-1/m}]$ is called \emph{isoclinic} if its leading coefficient $A_d$ is regular semisimple, i.e., $\mathrm{Stab}_W(A_d)=\{1\}$.   Likewise, a formal $G$-connection is called \emph{isoclinic} if its irregular class is isoclinic.

Our first goal is to determine the braid $\beta_A$ associated to an isoclinic irregular class. To this end, we need to recall some facts about regular elements in $W$.

\subsection{Springer elements}

Let $m$ be a regular number for $W$. By \cite[Proposition 4.10]{Springer}, there exist $w\in W$ and $x\in \ft$ such that
\[
w\cdot x=e^{2\pi i/m}x \qquad \text{and} \qquad \mathrm{Re}(\alpha(x))>0 \ \text{for all}\ \alpha\in \Phi^+.
\]
Following Brou\'{e} and Michel \cite{BM}, we call $w$ an \emph{$m$-Springer element}. For example, when $W=S_4$ with the standard Coxeter structure, $s_1s_3s_2$ is a Springer element while $s_1s_2s_3$ is not. We refer the reader to Appendix 1 of \op for a list of Springer elements in all types. 

Let $B_W$ denote the Artin braid group of $W$. 

\begin{thm} \label{t:Springer}
Suppose $w\in W$ is an $m$-Springer element.
\begin{enumerate} 
\item If $m$ is elliptic, then $w$ has minimal length in its conjugacy class. 
\item For all positive integers $d<m/2$, we have $\ell(w^d)=d\,\ell(w)=d\, |\Phi|/m$. 
\item We have $(\widetilde{w})^m=\Pi_W$.
\item Every $m^{\text{th}}$ root of $\Pi_W$ is $B_W$-conjugate to $\widetilde{w}$. 
\item Suppose $w'$ is an element of $W$ such that $\big(\widetilde{w'}\big)^m=\Pi_W$. Then $w$ and $w'$ are related by a cyclic shift. 
\end{enumerate}
\end{thm}

\begin{proof} Part (1) is proved in Proposition 4.10 of \cite{Springer}. Part (2) is proved in \cite[Proposition 3.9]{BM}. Part (3) is Proposition 3.12 of \cite{BM}. Part (4) is Theorem 12.4 of \cite{BessisKPi1}. Part (5) follows from Theorem 3.17 of \cite{BM}. 
\end{proof}

Let $W(m) \subseteq W$ denote the conjugacy class corresponding to a regular number $m$. We have seen that there exist elements of $W(m)$ which are not $m$-Springer with respect to our choice of positive roots. However, by \cite[Proposition 4.10]{Springer}, for every $w \in W(m)$, there exists a choice of positive roots with respect to which $w$ becomes an $m$-Springer element. In particular, for any $w \in W(m)$, there is a choice of positive roots such that the lift $\widetilde{w} \in B_W^+$ of $w$ with respect to this choice is an $m^{\mathrm{th}}$-root of $\Pi_W$.

\subsection{The braid associated to an isoclinic class}

Let $A$ be an isoclinic irregular class of slope $\nu=d/m$, where $\gcd(d,m)=1$. Then $m$ is a regular number for $W$ \cite[Lemma 2.2]{JY}. 
Let $w_m$ be an  $m$-Springer element. 

\begin{prop}
Up to cyclic shift, we have $\beta_A=(\widetilde{w_m})^d$. 
\end{prop}

If $m$ is elliptic, this result appears without proof in \cite[\S 4.4.1]{BBAMY}. The general case  is discussed in detail in \cite{braids}. For the sake of completeness, we give an outline of the proof. 

\begin{proof}[Outline of the proof]
 Let $A\in \ft[t^{-1/m}]$ be a toral irregular class. We first recall the definition of the conjugacy class in the braid group $B_W$ associated to $A$. Observe that we can think of $A$ as a map $\bC^\times \ra \ft^{\mathrm{reg}}/W$. Restricting this to a circle of radius $\epsilon$, we obtain an element of $\pi_1(\ft^{\mathrm{reg}}/W) = B_W$. Thus, $A$ gives rise to a conjugacy class $[\beta]\subset B_W$. One can show that this  class is independent of the circle (provided $\epsilon$ is small enough) and the base point. By construction, the positive braid $\beta_A$ belongs to $[\beta]$.
 
Now suppose $d=1$ and $A(t^{1/m})=xt^{-1/m}$ with $x\in \ft^{\mathrm{reg}}$. It is easy to see that the braid associated to the irregular class $xt^{-1}$ is the full twist $\Pi_W$. Thus, $\beta_A$ must be an $m^{\text{th}}$ root of $\Pi_W$. By Theorem \ref{t:Springer}.(5), up to a cyclic shift, we have $\beta_A=\widetilde{w_m}$. The proof for general $d$ is similar. 
\end{proof}

\subsection{Proof of Theorem \ref{t:main}} 
We now give the proof of our main theorem. The case $\nu>1$ follows from Lemma \ref{l:Full}. So, suppose $\nu<1$. Let $w$ be an $m$-Springer element. 
We follow Lusztig's strategy (discussed in \S \ref{ss:Nice}) and count points on $\cM((\widetilde{w})^d, \cC)$ over finite fields.   
We have the following generalisation of \eqref{eq:Lusztig}: 
\begin{equation} \label{eq:Trinh}
|\cM((\widetilde{w})^d, \cC)^F|
= \frac{1}{|G^F|} \sum_{g\in \cC^F} \mathrm{tr}((g, (T_w)^d), \cF) 
= \frac{1}{|G^F|} \sum_{g\in \cC^F}  \sum_{E\in \mathrm{Irr}(W)}\,  \mathrm{tr}((T_w)^d, E_q)\, \mathrm{tr}(g, \rho_E).
\end{equation} 
See, for instance,  \cite[Lemma 8.5.7]{Trinh} where a closely related formula appears. (Note, however, Trinh considers a slightly different action of $H_q$ on $\cF$.)

 We prove the theorem by explicitly computing the right hand side of \eqref{eq:Trinh} for every Springer element $w$ and every unipotent class $\cC$. The naive approach is to expand $(T_w)^d$ in terms of the standard basis of $H_q$ and compute the above value using the character table of $H_q$ and $G(\Fq)$. Indeed, one can verify the theorem in many cases by following this approach. However, for certain slopes in types $E_6$, $E_7$ and $E_8$, this approach becomes computationally unfeasible. 

We can dramatically simplify the computation using the following fact: 
\begin{lem} 
Let $\mathbf{c}(E_q):=|\Phi| - \mathbf{a}(E_q) - \mathbf{A}(E_q)$ be the content of $E_q$, where $\mathbf{a}$ and $\mathbf{A}$ are Lusztig's $\mathbf{a}$- and $\mathbf{A}$-function, respectively. Let $w$ be a Springer element. Then, we have: 
\[
\mathrm{tr}((T_w)^d, E_q) = q^{\nu\, \mathbf{c}(E_q)}\, \mathrm{tr}(w^d, E).
\] 
\end{lem} 

\begin{proof} 
See, for example, \cite[Corollary 4.20]{BM} and \cite[Corollary 9.2.2]{Trinh}. 
\end{proof} 

The above discussion implies  
\begin{equation} 
|\cM((\widetilde{w})^d, \cC)^F|
= \frac{1}{|G^F|} \sum_{g\in \cC^F}  \sum_{E\in \mathrm{Irr}(W)}\,  q^{\nu\, \mathbf{c}(E_q)}\, \mathrm{tr}(w^d, E)\, \mathrm{tr}(g, \rho_E). 
\end{equation} 
Observe that the summands involve only character values of $W$ and $G^F$. These quantities can be quickly computed using CHEVIE \cite{Michel}, allowing us to prove the theorem. 
We refer the reader to \cite{Whitbread} for the code computing the above sum and \S\ref{ss:examplecode} for examples.  \qed

\subsection{Rigid isoclinic connections}\label{ss:rigid}
In \cite[\S 7.9]{JY}, a table listing local data of \emph{potentially} rigid elliptic isoclinic connections is presented. However, the table was left incomplete, since the authors were only able to determine the existence of such connections in a limited number of cases. Using our main theorem, we can complete this table; see \S\ref{t:rigid}. Moreover, we can prove: 

\begin{prop} 
Let $G$ be an almost simple group of exceptional type. Then the cohomologically rigid elliptic isoclinic $G$-connections (whose local data is given in \S\ref{ss:rigidTable}) are physically rigid. 
\end{prop} 

\begin{proof} 
Let $\cC$ be a unipotent class such that the stack $\Conn(A, \cC)$ contains an elliptic cohomologically rigid connection. Let $\mathtt{Conn}(A, \cC)$ be its coarse moduli space. We would like to show that this variety has a unique $\bC$-point. Equivalently, we want to prove that the coarse moduli space  $\mathtt{M}(\beta_A, \cC)$ on the Betti side has a unique $\bC$-point. 
By \cite{KatzAppendix}, it suffices to show that $\mathtt{M}(\beta_A, \cC)$ has a unique $\mathbb{F}_q$-point. 
Now, it is known that (under the assumption that $w$ is elliptic) the action of $G$ on 
\[
\bigg\{
(g,F_1,\dots, F_{d+1})\in \cC\times\fB^{d+1}
\, \bigg|\, 
\substack{
\displaystyle gF_1=F_{d+1}\\[0.1cm]
\displaystyle w(F_i, F_{i+1})=w,\, i=1,\ldots,d  
}
\bigg\}
\]
has finite stabilisers. Thus, $\cM(\beta_A, \cC)$ is a Deligne--Mumford stack.\footnote{Here is an argument in the de Rham setting. It is known that an elliptic isoclinic connection is irreducible, cf.\ \cite[Lemma 17]{Yi}.  It follows that every connection in $\Conn(A, \cC)$  is irreducible. In particular, these connections have finite stabilisers. Therefore, $\Conn(A, \cC)$ is a Deligne--Mumford stack.}  As noted in \cite[\S 2]{stacks}, the number of points of a Deligne--Mumford stack over  a finite field is the same as the number of points of its coarse moduli space. Finally, our computer calculations show that in all cohomologically rigid elliptic cases, we have
$|\cM(\beta_A, \cC)^F|=1$.
\end{proof} 

We expect this result to hold for classical groups as well; see  \cite[Table 11]{JY} for the local data of the corresponding connections.

\subsection{Explicit realisation} \label{ss:Explicit}
Let $G$ be an almost simple algebraic group with Coxeter number $h$. Let $d<h$ be a positive integer coprime to $h$ and let $\nu=d/h$.  Choose a generator $x_\alpha$ for each root space $\fg_\alpha$. 
Let 
\[
N_d :=\sum_{\mathrm{ht}(\alpha)=-d}  x_\alpha, \qquad  E_d:= \sum_{\mathrm{ht}(\alpha)=h-d} x_\alpha,\qquad \text{and}\qquad 
\nabla_\nu:= d+ (N_d+t^{-1} E_d)\frac{dt}{t}. 
\]
Following \cite{KS-rigid}, we call $\nabla_\nu$ a Coxeter $G$-connection. Note that $\nabla_\nu$ is regular singular at $\infty$ with monodromy $\exp(N_d)$. As shown in \emph{op. cit}, it is irregular (and isoclinic) at $0$ with slope $d/h$.

\begin{prop} The connection $\nabla_\nu$ has minimal (local) monodromy among all Coxeter $G$-connections of slope $d/h$ and unipotent monodromy, i.e. $\exp(N_d)=\cC_{d/h}$. 
\end{prop} 

\begin{proof}
 When $G$ is of classical type, the partition corresponding to $N_d$ was computed in \cite[\S 4.3]{KS-rigid}. One readily verifies that the result agrees with the explicit description of $\cC_{d/h}$ given in Table 1 of \cite{JY}. 
When $G$ is of exceptional type, we have checked that $\exp(N_d)$ agrees with the class $\cC_{d/h}$ given in \S \ref{ss:minimal}, using a GAP package developed by de Graaf \cite{dG}. 
\end{proof}

We refer the reader to \cite{KS-galois} for a description of the differential Galois group of $\nabla_\nu$.

\fontsize{10}{12}\selectfont

\newpage

\appendix 
\section{{}}
\subsection{Minimal unipotent classes for isoclinic connections} The non-elliptic slopes are coloured blue.
\label{ss:minimal}
\begin{table}[H]
\begin{tabular}{|c|c|}
\multicolumn{1}{c}{$F_4$} & \multicolumn{1}{c}{} \\ 
\hline
$\nu$ & $\cC_\nu$ \\
\hline
$1/12$  & $F_4$                 \\
$5/12$  & $A_2+\widetilde{A_1}$ \\
$7/12$  & $A_1+\widetilde{A_1}$ \\
$11/12$ & $A_1$                 \\
$1/8$ & $F_4(a_1)$            \\
$3/8$ & $A_2+\widetilde{A_1}$ \\
$5/8$ & $\widetilde{A_1}$     \\
$7/8$ & $A_1$                 \\
$1/6$ & $F_4(a_2)$ \\
$5/6$ & $A_1$      \\
$1/4$ & $F_4(a_3)$ \\
$3/4$ & $A_1$      \\
$1/3$ & $\widetilde{A_2}+A_1$ \\
$2/3$ & $\widetilde{A_1}$     \\
$1/2$ & $A_1+\widetilde{A_1}$ \\
\hline
\multicolumn{1}{c}{} & \multicolumn{1}{c}{} \\ 
\multicolumn{1}{c}{$E_6$} & \multicolumn{1}{c}{} \\ 
\hline
$\nu$ & $\cC_\nu$ \\
\hline
$1/12$  & $E_6$     \\
$5/12$  & $A_2+2A_1$ \\
$7/12$  & $3A_1$     \\
$11/12$ & $A_1$      \\
$1/9$ & $E_6(a_1)$ \\
$2/9$ & $A_4+A_1$  \\
$4/9$ & $A_2+A_1$  \\
$5/9$ & $3A_1$     \\
$7/9$ & $A_1$      \\
$8/9$ & $A_1$      \\
\color{blue}$1/8$ & $D_5$  \\
\color{blue}$3/8$ & $A_2+2A_1$ \\
\color{blue}$5/8$ & $2A_1$ \\
\color{blue}$7/8$ & $A_1$ \\
$1/6$ & $E_6(a_3)$ \\
$5/6$ & $A_1$ \\
\color{blue}$1/4$ & $D_4(a_1)$  \\
\color{blue}$3/4$ & $A_1$ \\
$1/3$ & $2A_2+A_1$ \\
$2/3$ & $2A_1$     \\
\color{blue}$1/2$ & $3A_1$ \\
\hline
\end{tabular}
\quad\quad 
\begin{tabular}{|c|c|}
\multicolumn{1}{c}{$G_2$} & \multicolumn{1}{c}{} \\ 
\hline
$\nu$ & $\cC_\nu$ \\
\hline
$1/6$ & $G_2$ \\
$5/6$ & $A_1$ \\
$1/3$ & $G_2(a_1)$ \\
$2/3$ & $A_1$ \\
$1/2$ & $\widetilde{A_1}$ \\
\hline 
\multicolumn{1}{c}{} & \multicolumn{1}{c}{} \\ 
\multicolumn{1}{c}{$E_7$} & \multicolumn{1}{c}{} \\ 
\hline
$\nu$ & $\cC_\nu$ \\
\hline
$1/18$  & $E_7$        \\
$5/18$  & $A_3+A_2+A_1$ \\
$7/18$  & $A_2+3A_1$    \\
$11/18$ & $3A_1'$       \\
$13/18$ & $2A_1$        \\
$17/18$ & $A_1$         \\
$1/14$  & $E_7(a_1)$    \\
$3/14$  & $A_4+A_2$     \\
$5/14$  & $2A_2+A_1$    \\
$9/14$  & $2A_1$        \\
$11/14$ & $A_1$         \\
$13/14$ & $A_1$         \\
\color{blue}$1/9$ & $E_6(a_1)$ \\
\color{blue}$2/9$ & $A_4+A_1$ \\
\color{blue}$4/9$ & $A_2+A_1$ \\
\color{blue}$5/9$ & $(3A_1)'$ \\
\color{blue}$7/9$ & $A_1$ \\
\color{blue}$8/9$ & $A_1$ \\
\color{blue}$1/7$ & $A_6$ \\
\color{blue}$2/7$ & $A_3+A_2$ \\
\color{blue}$3/7$ & $A_2+2A_1$ \\
\color{blue}$4/7$ & $(3A_1)'$ \\
\color{blue}$5/7$ & $2A_1$ \\
\color{blue}$6/7$ & $A_1$ \\
$1/6$ & $E_7(a_5)$ \\
$5/6$ & $A_1$      \\
\color{blue}$1/3$ & $2A_2+A_1$ \\
\color{blue}$2/3$ & $2A_1$      \\
$1/2$ & $4A_1$ \\
\hline
\end{tabular}
\quad \quad 
\begin{tabular}{|c|c|}
\multicolumn{1}{c}{$E_8$} & \multicolumn{1}{c}{} \\ 
\hline
$\nu$ & $\cC_\nu$ \\
\hline
$1/30$  & $E_8$            \\
$7/30$  & $A_4+A_2+A_1$     \\
$11/30$ & $2A_2+2A_1$       \\
$13/30$ & $A_2+3A_1$        \\
$17/30$ & $4A_1$            \\
$19/30$ & $3A_1$            \\
$23/30$ & $2A_1$            \\
$29/30$ & $A_1$             \\
$1/24$  & $E_8(a_1)$        \\
$5/24$  & $A_4+A_3$         \\
$7/24$  & $A_3+A_2+A_1$     \\
$11/24$ & $A_2+2A_1$        \\
$13/24$ & $4A_1$            \\
$17/24$ & $2A_1$            \\
$19/24$ & $A_1$             \\
$23/24$ & $A_1$             \\
$1/20$  & $E_8(a_2)$        \\
$3/20$  & $A_6+A_1$         \\
$7/20$  & $2A_2+2A_1$       \\
$9/20$  & $A_2+2A_1$        \\
$11/20$ & $4A_1$            \\
$13/20$ & $2A_1$            \\
$17/20$ & $A_1$             \\
$19/20$ & $A_1$             \\
$1/15$  & $E_8(a_4)$        \\
$2/15$  & $D_7(a_2)$        \\
$4/15$  & $D_4(a_1)+A_2$    \\
$7/15$  & $A_2+A_1$         \\
$8/15$  & $4A_1$            \\
$11/15$ & $2A_1$            \\
$13/15$ & $A_1$             \\
$14/15$ & $A_1$             \\
\hline
\end{tabular}
\, 
\begin{tabular}{|c|c|}
\multicolumn{1}{c}{$E_8$} & \multicolumn{1}{c}{} \\ 
\hline
$\nu$ & $\cC_\nu$ \\
\hline
$1/12$  & $E_8(a_5)$        \\
$5/12$  & $A_2+3A_1$        \\
$7/12$  & $3A_1$            \\
$11/12$ & $A_1$             \\
$1/10$ & $E_8(a_6)$        \\
$3/10$ & $D_4(a_1)+A_1$    \\
$7/10$ & $2A_1$            \\
$9/10$ & $A_1$             \\
$1/8$ & $A_7$              \\
$3/8$ & $2A_2+A_1$         \\
$5/8$ & $3A_1$             \\
$7/8$ & $A_1$              \\
$1/6$ & $E_8(a_7)$          \\
$5/6$ & $A_1$               \\
$1/5$ & $A_4+A_3$           \\
$2/5$ & $A_2+3A_1$          \\
$3/5$ & $3A_1$              \\
$4/5$ & $A_1$               \\
$1/4$ & $2A_3$              \\
$3/4$ & $2A_1$              \\
$1/3$ & $2A_2+2A_1$          \\
$2/3$ & $2A_1$               \\
$1/2$ & $4A_1$               \\
\hline
\end{tabular}
\label{ta:unipotent}
\end{table}

\newpage

\subsection{Rigid elliptic isoclinic connections} \label{ss:rigidTable}
The entries undecided in \cite{JY} are coloured red.
\begin{table}[H]
\begin{center}
\begin{tabular}{|c|c|c|c|}
\hline
$G$ 	& $\nu$   	& $\cC$ & existence    		\\
\hline
$G_2$         		& $2/3$   	& $A_1$ 			& \color{red}yes                		\\
$F_4$         		& $3/4$   	& $A_1$ 			& \color{red}yes                		\\
${}$          		& $3/8$   	& $A_2+\tilde{A}_1$ 	& \color{red}yes		\\
${}$          		& $5/8$   	& $\tilde{A}_1$ 		& yes         	\\
$E_6$         		& $2/9$   	& $A_4+A_1$ 		& \color{red}yes            \\
${}$          		& $4/9$   	& $A_2+A_1$ 		& \color{red}yes           	\\
${}$          		& $7/9$   	& $A_1$ 			& \color{red}yes                	\\
${}$          		& $5/12$  	& $2A_2$ 			& no               	\\
$E_7$         		& $3/14$  	& $D_5(a_1)$ 		& \color{red}no           	\\
${}$          		& $3/14$  	& $A_4+A_2$ 		& \color{red}yes           	\\
${}$          		& $9/14$  	& $2A_1$ 			& \color{red}yes               \\
${}$          		& $11/14$ 	& $A_1$ 			& \color{red}yes               \\
${}$          		& $5/18$  	& $A_3+A_2$ 		& no           	\\
${}$          		& $7/18$  	& $2A_2$ 			& no            \\
${}$         		 & $7/18$  	& $A_2+3A_1$ 		& yes           	\\
\hline
\end{tabular}
\qquad
\begin{tabular}{|c|c|c|c|}
\hline
$G$ 	& $\nu$   	& $\cC$ & existence     		\\
\hline
$E_8$         		& $3/10$  	& $D_4(a_1)+A_1$ 	& \color{red}yes        	\\
${}$          		& $5/12$  	& $A_3$ 			& \color{red}no                		\\
${}$          		& $2/15$  	& $D_6$ 			& \color{red}no                		\\
${}$          		& $2/15$  	& $E_6$ 			& \color{red}no                		\\
${}$          		& $2/15$  	& $D_7(a_2)$ 		& \color{red}yes           	\\
${}$          		& $4/15$  	& $D_4+A_1$ 		& \color{red}no            	\\
${}$          		& $4/15$  	& $D_4(a_1)+A_2$ 	& \color{red}yes       	\\
${}$          		& $7/15$  	& $A_2+A_1$ 		& \color{red}yes            	\\
${}$          		& $3/20$  	& $A_6+A_1$ 		& \color{red}yes            	\\
${}$          		& $3/20$  	& $E_7(a_4)$ 		& \color{red}no           	\\
${}$          		& $7/20$  	& $A_3+A_1$ 		& \color{red}no            	\\
${}$          		& $13/20$ 	& $2A_1$ 			& \color{red}yes               	\\
${}$          		& $5/24$  	& $D_4+A_2$ 		& \color{red}no            	\\
${}$          		& $5/24$  	& $E_6(a_3)$ 		& \color{red}no           	\\
${}$          		& $7/24$  	& $A_3+A_2$ 		& \color{red}no            	\\
${}$          		& $19/24$ 	& $A_1$ 			& \color{red}yes                		\\
${}$          		& $7/30$  	& $A_4+2A_1$ 		& no           	\\
${}$          		& $17/30$ 	& $3A_1$ 			& no               	\\
\hline
\end{tabular}
\end{center}
\label{t:rigid}
\end{table}

\subsection{Examples of the computer code} \label{ss:examplecode}
We explain how to use the code found in the repository \cite{Whitbread}.

\subsubsection{Counting points and determining $\cC_\nu$}
We will determine $\cC_\nu$ for the slope $\nu=2/3$ in type $G_2$. To do so, we count points on the braid stack associated to the braid $\beta = (\widetilde{w})^2$ where $w$ is an $3$-Springer element. We choose the $3$-Springer element $w=\mathfrak{c}^2 = s_1s_2s_1s_2$. Then we use the command
\begin{center}
\begin{verbatim}
count_points(coxgroup(:G,2),[1,2,1,2],2)
\end{verbatim}
\end{center}
which prints the table below.

\begin{table}[h]
\begin{tabular}{|l|r|}
\hline
$\cC \subseteq G$ & $|\mathcal{M}(\beta,\cC)^F|$ \\
\hline
$1$ & $0$ \\
$A_1$ & $1$ \\
$\widetilde{A_1}$ & $q^2$ \\
$G_2(a_1)$ & $q^4$ \\
$G_2$ & $q^6$ \\
\hline
\end{tabular}
\end{table}
Visually inspecting the table, we see $\cM(\beta,\cC)\neq \emptyset$ if and only if $\cC\geq A_1$. Therefore $\cC_\nu = A_1$. Moreover, we have $|\cM(\beta,\cC_\nu)^F|=1$; thus, we have a physically rigid isoclinic connection of slope $\nu$ and monodromy $\cC_\nu$.

Alternatively, we can use the following functions to determine $\cC_\nu$ and $|\cM(\beta,\cC_\nu)^F|$:

\begin{enumerate}[\itshape(i)]
\item To calculate $\cC_\nu$, we use the command
\begin{center}
\begin{verbatim}
interval_reps(coxgroup(:G,2),[1,2,1,2],2)
\end{verbatim}
\end{center}
which returns the tuple of unipotent classes
\begin{center}
\begin{verbatim}
(UnipotentClass(A1), UnipotentClass(G2))
\end{verbatim}
\end{center}
representing the fact that $\cM(\beta,\cC)\neq \emptyset$ if and only if $\cC\in[A_1,G_2]$.
\item To calculate $|\cM(\beta,\cC_\nu)^F|$, we use the command
\begin{center}
\begin{verbatim}
count_points_lower(coxgroup(:G,2),[1,2,1,2],2)
\end{verbatim}
\end{center}
which returns the point-count $|\cM(\beta,\cC_\nu)^F|=1$.
\end{enumerate}

\subsubsection{A positive braid which is not nice}
We saw in \S\ref{ss:Nice} an example of an element $w\in W$ for which the positive braid $\widetilde{w}$ is not nice. This example was $w = s_2s_3s_2s_4s_3s_2s_3$ in type $F_4$. To verify $\widetilde{w}$ is not nice, we count points on the braid stack $\cM(\widetilde{w},\cC)$ using the command
\begin{center}
\begin{verbatim}
interval_reps(coxgroup(:F,4),[2,3,2,4,3,2,3],1)
\end{verbatim}
\end{center}
which prints the table below.
\begin{table}[h]
\begin{tabular}{|l|r|}
\hline
$\cC \subseteq G$ & $|\mathcal{M}(\beta,\cC)^F|$ \\
\hline
$1$ 						& $0$\\
$A_1$ 						& $0$\\
$\widetilde{A_1}$ 			& $0$\\
$A_1+\widetilde{A_1}$ 	& $0$\\
$\widetilde{A_2}$ 			& $q^{-7}\Phi_1^{-1}$\\
$A_2$ 						& $0$\\
$A_2+\widetilde{A_1}$ 	& $0$\\
$\widetilde{A_2}+A_1$ 	& $q^{-7}\Phi_3$\\
$B_2$ 						& $q^{-4}\Phi_1^{-1}$\\
$C_3(a_1)$ 				& $(q^2+3q+1)q^{-5}$\\
$F_4(a_3)$ 				& $q^{-5}\Phi_1\Phi_2^2$\\
$C_3$ 						& $(q^3+q^2+q-2)q^{-2}\Phi_1^{-1}$\\
$B_3$ 						& $q^{-2}\Phi_2$\\
$F_4(a_2)$ 				& $(q^3+q^{-2}+q-2)q^{-2}$\\
$F_4(a_1)$ 				& $(q^3+q^2-1)q^{-1}$\\
$F_4$ 						& $q^3$\\
\hline
\end{tabular}
\end{table}

Therefore, both $\cM(\widetilde{w},B_2)$ and $\cM(\widetilde{w},\widetilde{A_2})$ are non-empty and $\cM(\widetilde{w},A_1+\widetilde{A_1})$ is empty. One readily verifies that $\widetilde{w}$ is not nice because there does not exist an interval which includes $B_2$ and $\widetilde{A_2}$ but excludes $A_1+\widetilde{A_1}$. 

\vfill

\end{document}